\documentclass[letterpaper, 11pt]{article}

\usepackage[numbers]{natbib}
\usepackage{amssymb,amsmath,xcolor,graphicx,xspace,colortbl,rotating} %
\usepackage{subfigure}
\usepackage{appendix}  
\usepackage{bm} 
\usepackage{cancel} 
\usepackage{boxedminipage}  
\usepackage{color}  
\usepackage{endnotes}  
\usepackage{ragged2e}  
\usepackage[onehalfspacing]{setspace}  
\usepackage{tabulary}  
\usepackage{textcomp}  
\usepackage{varioref}  
\usepackage{graphicx}
\graphicspath{ {./images/} }
\usepackage[normalem]{ulem}
 
\usepackage[margin=1in]{geometry}
\newtheorem {theorem}{Theorem}

\newtheorem {lemma}{Lemma}

\newenvironment {proof}[1][Proof]{\noindent \textbf {#1.} }{\ \rule {0.5em}{0.5em}}

\makeatletter
\let\@fnsymbol\@arabic
\makeatother
\begin{document}

\newcommand{\newbl}[1]{{\color{blue}  #1}}
\newcommand{\ap}[1]{{\color{red} [Andres:  #1]}}

\title{The Principle of Optimality in Dynamic Programming: A Pedagogical Note}

\author{Bar Light\protect\thanks{Business School and Institute of Operations Research and Analytics, National University of Singapore, Singapore. e-mail: \textsf{barlight@nus.edu.sg} }}  
\maketitle

\thispagestyle{empty}

 \noindent \noindent \textsc{Abstract}:
\begin{quote}

The principle of optimality is a fundamental aspect of dynamic programming, which states that the optimal solution to a dynamic optimization problem can be found by combining the optimal solutions to its sub-problems. While this principle is generally applicable, it is often only taught for problems with finite or countable state spaces in order to sidestep measure-theoretic complexities. Therefore, it cannot be applied to classic models such as inventory management and dynamic pricing models that have continuous state spaces, and students may not be aware of the possible challenges involved in studying dynamic programming models with general state spaces. To address this, we provide conditions and a self-contained simple proof that establish when the principle of optimality for discounted dynamic programming is valid. These conditions shed light on the difficulties that may arise in the general state space case.  We provide examples from the literature that include the relatively involved case of universally measurable dynamic programming and the simple case of finite dynamic programming  where our main result can be applied to show that the principle of optimality holds. 

   
\end{quote}



\newpage

\section{Introduction} 

Dynamic programming is a powerful tool used in many areas, including operations research, engineering, artificial intelligence, mathematical finance,  and economics. One of the major concepts in dynamic programming is the principle of optimality (sometimes called Bellman's principle of optimality), which is a method for solving a dynamic optimization problem by  breaking it down into smaller sub-problems. While the principle of optimality holds under general conditions, proving it in the general case is quite involved and requires a non-trivial background in measure theory \citep{bertsekas1996stochastic}. In addition, some important functions that are studied in dynamic programming are not well behaved with respect to measurability. For example,   
it is well known, that even under the natural assumption that the dynamic optimization problem's primitives are Borel measurable, the value function need not be Borel measurable \citep{blackwell1965discounted}. 

To overcome the measurability difficulty, major textbooks that are used to teach the principle of optimality in dynamic programming, even for advanced graduate level courses, typically avoid measure-theoretic issues by assuming that the state space is finite or countable (e.g., \cite{bertsekas2012dynamic} and \cite{puterman2014markov}). This approach has two major drawbacks. First, in some of the canonical dynamic optimization models in operations research and economics like inventory management, consumption-savings problems, or dynamic pricing models, the state space is  continuous. Hence, the theory that is developed for countable state spaces does not apply to these canonical models. 
Second, students who are introduced to the important field of dynamic programming may not fully comprehend or acknowledge the challenges involved in studying dynamic programming models with continuous or other more  general state spaces such as Borel spaces.

In this note, we provide a self-contained proof for the principle of optimality and the existence of an $\epsilon$-optimal stationary policy for discounted  dynamic programming that requires minimal knowledge of measure theory. The key assumptions that guarantee the correctness of the principle of optimality and the existence of an $\epsilon$-optimal stationary policy are that the Bellman operator preserves  
 some form of measurability and that for every positive tolerance there is a feasible measurable selection that approximates the supremum of the Bellman operator uniformly within that tolerance. These assumptions are simple to state and provide insight into the potential measurability challenges in dynamic programming with general state spaces (such as the issue of measurable selection).  
 The focus of the paper's approach is the nature of these assumptions. Unlike typical results, where assumptions are made directly on the primitives of the dynamic programming model (e.g., continuity or specific forms of measurability), the paper's assumptions are instead placed on the properties of the Bellman operator itself. This allows us to provide a relatively simple proof that is similar to the proof for the finite state space case, but our result can be applied to much more involved cases, as demonstrated through  our examples. Importantly, while our assumptions hold trivially for bounded rewards and finite or countable state spaces, establishing them in the general case can be quite involved. Typically, proving these assumptions relies on pure measure-theoretic or topological arguments that are somewhat disconnected from the essence of dynamic programming itself. Hence, we provide a framework that is both accessible and helps to bridge the gap between simple and involved cases without straying  from the core concepts of dynamic programming.

  Our results are applicable to important cases that have been studied in the literature.  We provide seven examples from the literature where our main result can be applied to prove the principle of optimality and the existence of an $\epsilon$-optimal stationary policy, including the general universally measurable case \citep{bertsekas1996stochastic}, the upper semi continuous case 
 \citep{maitra1968discounted}, the monotone case \cite{topkis1998supermodularity,light2021stochastic}, and the standard countable case  \citep{bertsekas2012dynamic}. We demonstrate a potential failure of our assumptions using an example from \cite{blackwell1965discounted} that involves Borel measurable primitives. The rest of the note is organized as follows.  In Section 2 we present the discounted dynamic programming model. In Section 2.1 we present the main theorem and its proof. In Section 2.2 we present the applications of the main Theorem.

\section{Discounted Dynamic Programming}

For the sake of brevity, we will focus on discounted dynamic programs in this note.

We define a discounted dynamic programming model in terms of a tuple of elements $(S ,A ,\Gamma  ,p , U , r , \beta  )$. 
The state space $S$ is a Borel subset of a Polish space that describes the possible states of the system. The action space $A$ is a Borel subset of a Polish space. Both spaces carry their relative topologies and Borel sigma-algebras. We denote by $\mathcal{B}(X)$ the Borel sigma-algebra on a Borel space $X$.

The set $U:=U(X,Y)$ is a set of functions from a Borel space $X$ to a Borel space $Y$, $f:X \rightarrow Y$,  such that $f$ is $\mathcal{U}$-measurable in the sense that $f^{-1}(B) \in \mathcal{U}(X)$ for all $B \in \mathcal{B}(Y)$,  where $\mathcal{U}(X)$ is a sigma-algebra that is well defined for every Borel space. We will assume in this note that $\mathcal{U}(X)$ consists of the Borel measurable sets or the universally measurable sets. For finite or countable $X$, both choices consist of all subsets of $X$. We choose to focus on these sets because they are widely studied in the literature on dynamic optimization (see Section \ref{Sec:applications}). Other choices require the corresponding integration and closure properties used in Lemma \ref{lemma:T-lambda-closed}.    Thus, $U(X,Y)$ is the set of all Borel measurable functions or all universally measurable functions from $X$ to $Y$; when $X$ is finite or countable, either choice includes all functions from $X$ to $Y$. The function 
$p :S \times A \times \mathcal{B}(S) \rightarrow [0 ,1]$ is a transition probability function, i.e., $p(s ,a , \cdot)$ is a probability measure on $(S,\mathcal{B}(S))$ for each $(s ,a) \in S \times A$ and $p(\cdot,\cdot,B)$ is a $\mathcal{U}$-measurable function for each $B \in \mathcal{B}(S)$. Integrals and expectations of universally measurable functions are taken with respect to the completions of the corresponding probability measures. The function 
$r : S \times A \rightarrow \mathbb{R}$ is the single-period 
  payoff function that is bounded from below, and $\mathcal{U}$-measurable. $0 <\beta  < 1$ is the discount factor. The constraint set  $\Gamma$ is a Borel measurable subset of $S \times A$ such that for any $s \in S$ the $s$-section $\Gamma(s)$ of $\Gamma$ is a  non-empty Borel measurable set of feasible actions in state $s \in S$.

The process starts at some state $s(1) \in S$. Suppose that at time $t$ the state is $s(t)$. Based on $s(t)$, the decision maker (DM) chooses an action $a(t) \in \Gamma(s(t))$ and receives a payoff $r(s(t) ,a(t))$. The probability that the next period's state $s(t +1)$ will lie in $B \in \mathcal{B}(S)$ is given by $p(s(t) ,a(t) ,B)$.

 Let $H^{t}:= S \times (A \times S)^{t-1} $ be the set of finite histories up to time $t$. A strategy $\sigma$ is a sequence $(\sigma _{1} ,\sigma _{2} , \ldots )$ of  functions $\sigma _{t} \in U(H^{t},A)$ such that $\sigma_{t}(s(1),a(1),\ldots,s(t)) \in \Gamma(s(t))$ for all $t$ and all $(s(1),a(1),\ldots,s(t)) \in H^{t}$. Each initial state $s(1) \in S$, a strategy $\sigma $, and a transition probability function $p$ induce a well-defined probability measure over the measurable space of all infinite histories $(H^{\infty },\mathcal{B}(H^{\infty }))$ (see Propositions 7.28 and 7.45 in \cite{bertsekas1996stochastic}). 
 We denote the expectation with respect to that probability measure by $\mathbb{E}_{\sigma }$, and the associated stochastic process by $\{s(t) ,a(t)\}_{t =1}^{\infty }$. 
The decision maker's goal is to find a strategy that maximizes her expected discounted payoff. When the decision maker follows a strategy $\sigma$ and the initial state is $s \in S$, the expected discounted payoff is given by
\begin{equation} \label{eq: V_sigma}
V_{\sigma}(s) =\mathbb{E}_{\sigma }\sum \limits_{t =1}^{\infty }\beta^{t -1}r(s(t),a(t)).
\end{equation}
Define 
\begin{equation*} 
V(s) =\sup _{\sigma }V_{\sigma }(s) .
\end{equation*}
We call $V$ the value function; under the assumptions of Theorem \ref{Thm:mainDP}, it is real-valued and bounded. A strategy $\sigma$ is said to be optimal if $V_{\sigma} (s) = V(s)$ for all $s$. For a fixed $\epsilon>0$, a strategy $\sigma$ is said to be $\epsilon$-optimal if for every strategy $\sigma'$ and every $s\in S$ we have $V_{\sigma'}(s) \leq V_{\sigma}(s) + \epsilon$. A strategy is said to be stationary if $\sigma_t(h(t)) = \lambda(s(t))$ where $h(t) = (s(1),a(1),\ldots,s(t) )$ for any history $h(t) \in H^{t}$ and some function $\lambda \in U(S,A)$ satisfying $\lambda(s)\in\Gamma(s)$ for every $s$, which is called a stationary policy. 

For a bounded function $f \in U(S,\mathbb{R})$ we define the operator (called the Bellman operator) $T$ by 
\begin{equation*}Tf(s) =\sup _{a \in \Gamma(s)}Q(s ,a ,f), 
\end{equation*}
where 
\begin{equation}Q(s ,a ,f) = r(s ,a) +\beta \int _{S}f(s^{ \prime })p(s ,a ,ds^{ \prime }) \label{eq:h}
\end{equation}
and, for $c\geq 0$, the set of feasible measurable selections within $c$ of the supremum by
\begin{equation*}
\Lambda_f^c
=\left\{\lambda\in U(S,A):
\lambda(s)\in\Gamma(s),\quad
Q(s,\lambda(s),f)\geq Tf(s)-c
\text{ for all }s\in S\right\}.
\end{equation*}
Thus, $\Lambda_f^0$ consists of the feasible measurable selections that attain the supremum.

 For a stationary policy $\lambda \in U(S,A)$ and a bounded $f \in U(S,\mathbb{R})$ we also define $T_{\lambda}$ by
\begin{equation}
    T_{\lambda}f(s) =  r(s,\lambda(s)) + \beta \int _{S} f(s') p(s,\lambda(s),ds') = Q(s,\lambda(s),f).
\end{equation}

\subsection{The Dynamic Programming Principle} 

Let $B(S)$ be the space  of all real-valued bounded  functions on $S$ equipped with the sup-norm $ \Vert f \Vert = \sup _{s \in S} |f(s)| $ and the induced metric $d_{\infty} (f,g):= \sup _{s \in S} |f(s) - g(s)|$  for $f,g \in B(S)$. Let $B_{m}(S) =  B(S) \cap U(S,\mathbb{R})$.  We now present the main theorem of this note that shows that the value function satisfies the Bellman equation under certain conditions.  We show that if there exists a closed set $D$ of $B_{m}(S)$ containing the constant functions specified in Theorem \ref{Thm:mainDP}, such that $T(D) \subseteq D$ and for each $f \in D$ and $c>0$ we have $\Lambda_{f}^{c} \neq \emptyset$, then the value function is in $D$, satisfies the Bellman equation, and for every $\epsilon>0$ there exists an $\epsilon$-optimal stationary policy. Extensions to unbounded rewards or undiscounted problems require additional assumptions and may use monotone convergence arguments (see \cite{bertsekas1996stochastic}, \cite{feinberg2012average} and \cite{feinberg2016partially}).

The assumption that $f \in D$ implies that $Tf \in D$  guarantees that the Bellman equation preserves some form of measurability as $D$ consists of appropriately chosen measurable functions. The assumption that $\Lambda_f^c\neq\emptyset$ for every $f\in D$ and $c>0$ guarantees a feasible measurable selection within any prescribed positive tolerance of the supremum, uniformly across states. In the finite and countable cases with bounded rewards these assumptions trivially hold (see Section \ref{Sec:applications}). In  the general case, these assumptions can be quite involved to establish. Finding  
 conditions on the model's primitives that guarantee that these assumptions hold are crucial in the study of dynamic programming. We discuss some important cases from the literature in Section \ref{Sec:applications}.

\begin{theorem} \label{Thm:mainDP} Let $D \subseteq B_{m}(S)$ be a complete metric space with the metric $d_{\infty}$ that contains the constant $0$ function and any constant function $M$ such that $M= \sup_{s \in S} f(s)$ for $f \in D$. 

Assume that $Tf\in D$ and $\Lambda_f^c\neq\emptyset$ for every $f\in D$ and every $c>0$.  

Then the following holds: 

1. $V$ is the unique function in $D$ that satisfies the Bellman equation, i.e., $TV = V$.

2. For every $\epsilon>0$, there is a stationary policy $\lambda \in \Lambda_{V}^{\epsilon(1-\beta)}$ that is $\epsilon$-optimal. 

3. Every $\lambda \in \Lambda_{V}^{0}$ is an optimal stationary policy. 

\end{theorem}

Note that $B_m(S)$ is closed in $B(S)$ under uniform convergence, so $B_m(S)$ is a complete metric space.

We first prove a few simple properties of the operators $T$ and $T_{\lambda}$ under the assumptions of Theorem \ref{Thm:mainDP}. Since $0\in D$ and $T0\in D$, we have $r(s,a)\leq\sup_{s\in S}T0(s)<\infty$ for all $a\in\Gamma(s)$. Together with the assumed lower bound, this shows that rewards are uniformly bounded on feasible state-action pairs. Also, $\Lambda_0^1\neq\emptyset$, so there is at least one strategy. 

\begin{lemma} \label{lemma:T-lambda-closed}
For each stationary policy $\lambda$, $f\in B_m(S)$ implies $T_\lambda f\in B_m(S)$. In addition, $V_\sigma\in B_m(S)$ for each strategy $\sigma$.
\end{lemma}
\begin{proof}
Let $f\in B_m(S)$ and let $\lambda\in U(S,A)$ be a stationary policy. The function $\int_S f(s')p(s,a,ds')$ is universally measurable (Borel measurable) when $f$ and $p$ are universally measurable (Borel measurable). Hence, $Q(s,a,f)$ has the same measurability property, since $r$ is $\mathcal U$-measurable. Thus, $T_\lambda f=Q(s,\lambda(s),f)$ is $\mathcal U$-measurable by closure under composition (see Chapter 7 in \cite{bertsekas1996stochastic}). The bound on feasible rewards shows that $T_\lambda f$ is bounded.

Applying the same integration and composition properties to finite histories shows that the expected payoff over each finite horizon is $\mathcal U$-measurable as a function of the initial state. Since feasible rewards are uniformly bounded, these payoffs converge uniformly to $V_\sigma$. Closure under pointwise limits therefore gives $V_\sigma\in B_m(S)$.
\end{proof}

For two functions $f,g$ from $S$ to $\mathbb{R}$, $f \geq g$ means $f(s) \geq g(s)$ for each $s \in S$.
\begin{lemma} \label{lemm:Tf-prop}
For $f,g\in B_m(S)$, the operator $T$ satisfies the following two properties.

(a) Monotone: $Tf \geq Tg$ whenever $f \geq g$. 

(b) Discounting: $T(f+c) = Tf + \beta c$ for any constant function $c$. 

In addition, $T_{\lambda}$ satisfies $(a)$ and $(b)$ for any stationary policy $\lambda \in U(S,A)$. 
\end{lemma}

\begin{proof}
    (a) We have $Q(s,a,f) \geq Q(s,a,g)$ whenever $f \geq g$. Hence, $Tf \geq Tg$.

    (b) For each $s \in S$ we have
    \begin{align*}
    T(f+c)(s)
    &= \sup_{a\in\Gamma(s)}
       \left\{r(s,a)+\beta\int_S f(s')p(s,a,ds')+\beta c\right\} \\
    &= Tf(s)+\beta c.
    \end{align*}

   The proof for the operator $T_{\lambda}$ is similar. 
\end{proof}

A key property of the operators $T$ and $T_{\lambda}$ is the contraction property. For $0\leq L<1$, the operator $T:D\to D$ is an $L$-contraction if $\Vert Tf - Tg \Vert \leq L \Vert f - g \Vert$ for all $f,g \in D$. 

\begin{lemma} \label{lemm:Tf-contraction}
    The operator $T:D\to D$ is a $\beta$-contraction, and $T_\lambda:B_m(S)\to B_m(S)$ is a $\beta$-contraction for every stationary policy $\lambda$. 
\end{lemma}

\begin{proof}
Let $f,g\in B_m(S)$ and let $\lambda\in U(S,A)$ be a stationary policy. By Lemma \ref{lemma:T-lambda-closed}, $T_\lambda$ maps $B_m(S)$ into itself. 

For any $s \in S$ we have 
\begin{align*}
|T_\lambda f(s)-T_\lambda g(s)|
&= \left|\beta\int_S(f(s')-g(s'))p(s,\lambda(s),ds')\right| \\
&\leq \beta\int_S|f(s')-g(s')|p(s,\lambda(s),ds') \\
&\leq \beta\Vert f-g\Vert.
\end{align*}
Taking the supremum over $s \in S$ yields $\Vert T_{\lambda}f - T_{\lambda} g \Vert \leq \beta \Vert f - g \Vert$ so $T_{\lambda}$ is a $\beta$-contraction on $B_{m}(S)$.

The same integral estimate gives, for all $s \in S$ and $a \in \Gamma(s)$, 
$$Q(s,a,f)  \leq Q(s,a,g) + \beta \Vert f - g \Vert  \leq Tg (s)+ \beta \Vert f - g \Vert.  $$
Thus, taking the supremum over $a \in \Gamma(s)$ yields 
$$ Tf(s) \leq Tg(s) + \beta \Vert f - g \Vert.$$
By a symmetrical argument, $Tg(s)\leq Tf(s)+\beta\Vert f-g\Vert$ for all $s\in S$. Restricting $f,g$ to $D$ and using $T(D)\subseteq D$ shows that $T:D\to D$ is a $\beta$-contraction. 
\end{proof}

\begin{proof} [Proof of Theorem \ref{Thm:mainDP}]
From the Banach fixed-point theorem (see Theorem 3.48 in \cite{aliprantis2006infinite}) and Lemma \ref{lemm:Tf-contraction} the operator $T:D \rightarrow D$ has a unique fixed-point $x$. We will show that $V=x$.  We first show that $V \leq x$. 

Let $k(s)=\sup_{a\in\Gamma(s)}r(s,a)=T0(s)$. Since $0\in D$ and $T(D)\subseteq D$, the function $k$ is bounded and belongs to $D$. Thus, $ M = \sup_{s \in S} k(s) $ belongs to $D$ also.  Using the same argument again for the constant function $f_{2}(s) = M$ for all $s \in S$ we have $Tf_{2}(s) = k(s) + \beta M $ is in $D$ so $M + \beta M$ is in $D$. Continuing inductively, we conclude that the constant function $M_{T}$ given by $ M_{T} =\sum _{t=1}^{T} \beta ^{t-1} M$ is in $D$. Because $D$ is complete, it is closed in $B_m(S)$, so the constant function $ y(s)= \lim_{T \rightarrow \infty} M_{T} = M /(1-\beta)$  for all $s \in S$ belongs to $D$.  Because feasible rewards are bounded above by $M$ we have $ V \leq y$. Hence, there exists a $y \in D $ such that $V \leq y$.

 For a strategy $\sigma$ let $\sigma | (s,a)$ be the strategy that is induced by $\sigma$ given that the first period's state-action pair was $(s,a)$, i.e., $\sigma | (s,a) _{t} (h_{t}) = \sigma _{t+1} (s,a,h_{t}) $ for all $t$ and any history $h_{t} \in H^{t}$. This is again a strategy, since taking a section preserves Borel and universal measurability.  
For any strategy $\sigma$, a function $y \in D$ with $V \leq y$, and $s \in S$ we have 
\begin{align*}
    V_{\sigma}(s) & = \mathbb{E}_{\sigma } \left ( r(s,a(1)) + \sum \limits_{t =2}^{\infty }\beta^{t -1}r(s(t),a(t)) \right )   \\
    & = r(s,\sigma_{1}(s)) + \beta \int _{S} V_{\sigma | (s,\sigma_{1}(s))} (s') p(s, \sigma_{1}(s),ds') \\
    & \leq  r(s,\sigma_{1}(s)) + \beta \int _{S} y (s') p(s, \sigma_{1}(s),ds') \\
    & \leq \sup_{a\in\Gamma(s)}\left\{r(s,a)+\beta\int_S y(s')p(s,a,ds')\right\} \\
    & = Ty(s). 
\end{align*}
   Thus, taking the supremum over all strategies $\sigma$ implies that $V \leq Ty$.  Repeating the same argument, we deduce that $V \leq T^{n}y$ for all $n \geq 1$. From the Banach fixed-point theorem $T^{n}y$ converges to the unique fixed-point of $T$ on $D$. We conclude that $V \leq x$.  

We now show that $V \geq x$. 
   Let $\epsilon>0$. Because $x\in D$, the theorem's assumption gives a stationary policy $\lambda\in\Lambda_x^{\epsilon(1-\beta)}$. Since $Tx=x$, we have $T_\lambda x\geq x-\epsilon(1-\beta)$. 
   
   We will now show that $T_{\lambda}^{n}x \geq x - \epsilon$ for all $n$. For $n=1$ the result holds from the construction of $\lambda$. Assume it holds for some $n\geq 1$. Then 
   $$ T_{\lambda}^{n+1}x = T_{\lambda} (T_{\lambda}^{n}x) \geq T_{\lambda} (x - \epsilon) = T_{\lambda}x - \beta \epsilon \geq  x - \epsilon$$
   where we used Lemma \ref{lemm:Tf-prop} to derive the inequalities. 
Thus, $T_{\lambda}^{n}x \geq x - \epsilon$ for all $n$. From Lemma \ref{lemm:Tf-contraction}  and the Banach fixed-point theorem  
 $T_{\lambda}^{n}x$ converges uniformly to the unique fixed-point of $T_{\lambda}$ in $B_{m}(S)$ which we denote by $W$ so $W \geq x - \epsilon$. So $W:S \rightarrow \mathbb{R}$ is bounded and $\mathcal{U}$-measurable. Iterating $W=T_\lambda W$ and letting the bounded discounted terminal term tend to zero gives 
\begin{align*}
W(s)
&=r(s,\lambda(s))+\beta\int_S W(s')p(s,\lambda(s),ds') \\
&=\mathbb E_\sigma\sum_{t=1}^\infty\beta^{t-1}r(s(t),a(t))
\end{align*}
with $s=s(1)$ and $\sigma$ is the stationary plan that plays according to $\lambda$, i.e., $\sigma_t(h_t)=\lambda(s(t))$ for all $h_t=(s(1),a(1),\ldots,s(t))$. Hence, $V\geq W\geq x-\epsilon$. Since $\epsilon>0$ was arbitrary, $V\geq x$.  We conclude that $V=x$ so $V$ is the unique fixed point of $T$ on $D$ and the selected $\lambda\in\Lambda_V^{\epsilon(1-\beta)}$ is an $\epsilon$-optimal stationary policy.  

In addition, if $\lambda\in\Lambda_V^0$, then 
$T_\lambda V=V$, so uniqueness of the fixed point of $T_\lambda$ implies that $\lambda$ is an optimal stationary policy.
\end{proof}

Under the conditions of Theorem \ref{Thm:mainDP}, $V$ belongs to the set $D$.  Hence, we can use Theorem \ref{Thm:mainDP} to prove properties of the value function such as monotonicity or concavity as we show in Section \ref{Sec:applications}.

\subsection{Applications of Theorem \ref{Thm:mainDP}}    \label{Sec:applications}
As a starting point, it is natural to assume that the dynamic programming problem's primitives  are Borel measurable and  choose $U$ to be the Borel measurable functions and $D$ to be the bounded Borel measurable functions on $S$. But there is a well-known measurability issue with this choice. The following example with Borel primitives is provided by \cite{blackwell1965discounted}. 
Let $S=A=[0,1]$ and $\Gamma(s)=A$ for every $s$. Then there is a Borel measurable set $B \subseteq S \times A$ such that its projection $E$ to $S$ is not Borel measurable.  Consider $f$ to be the zero function and $r = 1_{B}(s,a)$. Then $f$ and $r$ are Borel measurable. But $Tf (s) = \sup_{a \in A} 1_{B}(s,a) = 1_{E}(s)$ is not Borel measurable. Hence, we cannot apply Theorem \ref{Thm:mainDP} for this case. 

We now present the assumptions for some of the main cases studied in the literature on discounted dynamic programming, where Theorem 
\ref{Thm:mainDP} can be applied.

We say that $p$ is continuous  (concave, increasing, Borel measurable) if $\int_{S} f(s')p(s,a,ds')$ is continuous (concave, increasing, Borel measurable) on $S \times A$ for every bounded Borel measurable function $f$ on $S$ that is, respectively, continuous, concave, increasing, or Borel measurable. 
\begin{itemize}
    \item  \textbf{The upper semianalytic case.} The set  $U(X,Y)$ is the set of universally measurable functions from $X$ to $Y$ (see \cite{bertsekas1996stochastic} for a detailed study and definitions of analytic sets, upper semianalytic functions, and universally measurable sets and functions).   The payoff function $r$ is a bounded upper semianalytic function, and $p$ is Borel measurable. We choose $D$ to be the set of all bounded upper semianalytic functions on $S$. The set $D$ is closed in $B_{m}(S)$ and it can be shown that $f \in D$ implies $Tf \in D$ and $\Lambda_f^c\neq\emptyset$ for every $c>0$ (see Propositions 7.47, 7.48, and 7.50 in \cite{bertsekas1996stochastic}). That is, $Tf$ is upper semianalytic and there exists a universally measurable selection within any prescribed positive tolerance of its supremum. Hence, we can apply Theorem \ref{Thm:mainDP}.  

For the next example, recall that a function $f:S \rightarrow \mathbb{R}$ is upper semi continuous (u.s.c) if $\limsup f(s_{n}) \leq f(s)$ whenever $s_{n} \rightarrow s$.  The correspondence $\Gamma:S \rightarrow 2^{A}$ is u.s.c  if $\{s \in S: \Gamma(s) \subseteq U \}$ is an open set whenever $U$ is an open set in $A$ and lower semi continuous (l.s.c) if $\{s \in S: \Gamma(s) \cap U \neq \emptyset \}$ is an open set whenever $U$ is an open set in $A$ (see Chapter 17 in \cite{aliprantis2006infinite}).

\item \textbf{The upper semi continuous case.} The set $U(X,Y)$ is the set of Borel measurable functions from $X$ to $Y$. The set $\Gamma(s)$ is compact for each $s \in S$, $\Gamma$ is u.s.c, $r$ is bounded and u.s.c, and $p$ is continuous. Under these conditions, we can choose $D$ to be the set of all bounded u.s.c functions on $S$. The set $D$ is closed in $B_{m}(S)$ and it can be shown that $f \in D$ implies $Tf \in D$ and $\Lambda_f^0\neq\emptyset$ (see \cite{dubins1965gamble} and \cite{maitra1968discounted} for proofs of these claims and Section 2.3 in \cite{light2024course} for a textbook level treatment of the upper semi continuous case). In particular, the value function $V$ is u.s.c and there is a Borel measurable selection that achieves the supremum. Hence, we can apply Theorem \ref{Thm:mainDP}.

\item \textbf{The continuous case.} The set  $U(X,Y)$ is the set of Borel measurable functions from $X$ to $Y$. The set $\Gamma(s)$ is compact for each $s \in S$, $\Gamma$ is u.s.c and l.s.c, $r$ is bounded and continuous, $p$ is continuous. We can choose $D$ to be the set of all bounded and continuous functions on $S$ which is closed in $B_{m}(S)$. It can be shown that if $f \in D$ then $Tf \in D$  (see Theorem 17.31 in \cite{aliprantis2006infinite}) and $\Lambda_f^0\neq\emptyset$ from the u.s.c case. In particular, the value function $V$ is continuous. Hence, we can apply Theorem \ref{Thm:mainDP}.

For the concave case, $S$ and $A$ are subsets of real vector spaces. The correspondence $\Gamma$ is called convex if  $a_{1} \in \Gamma (s_{1})$ and $a_{2} \in \Gamma (s_{2})$ imply that $a_{\lambda} \in \Gamma (s_{\lambda})$ for all $s_{1},s_{2}$ and $\lambda \in (0,1)$ where   
 $s_{\lambda} = \lambda s_{1} +(1-\lambda) s_{2}$ and $a_{\lambda} = \lambda a_{1} +(1-\lambda) a_{2}$.
 
\item \textbf{The concave case.} The set $U(X,Y)$ is the set of Borel measurable functions from $X$ to $Y$. The set $S$ is convex, $\Gamma$ is convex, $r$ is concave and is strictly concave in $a$, $p$ is concave, and all the assumptions of the u.s.c case hold. We can choose $D$ to be the set of all bounded, u.s.c and concave functions on $S$ which is closed in $B_{m}(S)$.   It can be shown that if $f \in D$ then $Tf \in D$ (see the u.s.c case and Theorem 3.1 in \cite{light2024course} for concavity)  and $\Lambda_f^0\neq\emptyset$ from the u.s.c case. In particular, the value function $V$ is concave and the optimal action is unique at every state. Hence, we can apply Theorem \ref{Thm:mainDP}.  

We need the following definitions for the monotone case. A lattice $X$ is a partially ordered set such that each two elements $x,y \in X$ has a least upper bound (denoted by $x \lor y$) and a greatest lower bound (denoted by $x \land y$). 
Recall that a function $f:X \rightarrow \mathbb{R}$, where $X$ is a lattice in $\mathbb{R}^{n}$, (we assume the standard product order on $\mathbb{R}^{n}$, i.e., $x \geq y$, $x,y \in \mathbb{R}^{n}$ if $x_{i} \geq y_{i}$ for each $i=1,\ldots,n$) is supermodular if $f (x \vee y) + f(x \wedge y) \geq f(x) + f(y) $ for all $x,y \in X$ where as usual $x \vee y$ denotes the componentwise maximum and $x \wedge y$ the componentwise minimum. A correspondence $\Lambda$ from $X$ to $E$ where $\Lambda(x)$ is a (sub)-lattice for each $x$, is ascending if $\Lambda (x_{2}) \succeq_{set} \Lambda(x_{1})$ whenever $x_{2} \geq x_{1}$ where for two sets $D$ and $E$, write $D \succeq _{set} E$ if $x \in D$ and $x' \in E$ imply $x \lor x' \in D$ and $x \land x' \in E$ (see \cite{topkis1998supermodularity} for a comprehensive treatment of supermodularity). 
\item \textbf{The monotone case.} (i) (Monotone value function). The set $U(X,Y)$ is the set of Borel measurable functions  from $X$ to $Y$. The set $A \subseteq \mathbb{R}^{n}$ and $S\subseteq \mathbb{R}^{m}$. We have $\Gamma(s_{1}) \subseteq \Gamma(s_{2})$ whenever $s_{1} \leq s_{2}$, $r$ is increasing in $s$, $p$ is increasing in $s$, and all the assumptions of the u.s.c case hold. We can choose $D$ to be the set of all u.s.c, bounded, increasing functions which is closed in $B_{m}(S)$.  It can be shown that if $f \in D$ then $Tf \in D$ (see Lemma 3.9.4 in \cite{topkis1998supermodularity} and Theorem 3.1 in \cite{light2024course})  and $\Lambda_f^0$ is non-empty from the u.s.c case. In particular, the value function is increasing. Hence, we can apply Theorem \ref{Thm:mainDP}.

(ii) (Monotone policy function). We say that $p$ is supermodular if $\int_S f(s')p(s,a,ds')$ is supermodular on $S\times A$ for every bounded Borel measurable supermodular function $f$ on $S$. The set $U(X,Y)$ is the set of Borel measurable functions  from $X$ to $Y$. The sets $A \subseteq \mathbb{R}^{n}$ and $S\subseteq \mathbb{R}^{m}$ are lattices. The graph of $\Gamma$, $\operatorname{Gr} \Gamma = \{ (s,a) \in S \times A : a \in \Gamma(s) \}$ is a sub-lattice in $S \times A$, $r$ is supermodular, $p$ is supermodular, and all the assumptions of the u.s.c case hold. We can choose $D$ to be the set of all u.s.c, bounded, and supermodular functions which is closed in $B_{m}(S)$.  It can be shown that if $f \in D$ then $Tf \in D$, $\Lambda_f^0$ is non-empty, and the set of optimal solutions  \begin{equation*}
G(s)=\operatorname{argmax}_{a\in\Gamma(s)}
\left\{r(s,a)+\beta\int_S V(s')p(s,a,ds')\right\}
\end{equation*} is ascending (see Theorem 3.4 in \cite{light2024course} and see Theorem 3.9.2 in  \cite{topkis1998supermodularity} and Theorem 3 in \cite{light2021stochastic} for related results).  In particular, the value function $V$ is supermodular. Hence, we can apply Theorem \ref{Thm:mainDP}.

\item \textbf{The countable-compact case}. The set $U(X,Y)$ is the set of all Borel measurable functions from $X$ to $Y$. The set $S$ is countable, $\Gamma(s)$ is compact in $\mathbb{R}^{n}$ (a prominent choice is $\Gamma(s) = A $ for each $s$,  and $A$ is the set of all probability measures  on some finite action set $A'$), $r$ is  bounded. We can choose $D=B_m(S)=B(S)$. For every $f\in D$ and $c>0$, choose an action within $c$ of the supremum at each state. The resulting policy is Borel measurable because $S$ is countable, so $\Lambda_f^c\neq\emptyset$. Boundedness of $r$ also gives $T(D)\subseteq D$. Thus, Theorem \ref{Thm:mainDP} applies and gives an $\epsilon$-optimal stationary policy for every $\epsilon>0$. 

\item \textbf{The finite case.} The sets $S$ and $A$ are finite. We take $U$ to be the Borel measurable functions; on the finite state and history spaces this includes all functions. We can choose $D=B_m(S)=B(S)$, and $\Lambda_f^0\neq\emptyset$ for every $f\in D$. Thus, the conditions of Theorem \ref{Thm:mainDP} hold and an optimal stationary policy exists. 

\end{itemize}

\bibliographystyle{ecta}
\bibliography{DP}

\end{document}